\documentclass[11pt]{article}
\usepackage{amsthm, amsmath, amssymb, amsfonts, url, booktabs, tikz, setspace, fancyhdr, bm}
\usepackage{hyperref}
\usepackage{geometry}
\usepackage{hyperref, enumerate}
\usepackage[shortlabels]{enumitem}
\usepackage[babel]{microtype}
\usepackage[english]{babel}
\usepackage[capitalise]{cleveref}
\usepackage{comment}
\usepackage{bbm}
\usepackage{csquotes}
\usepackage{mathabx}
\usepackage{tikz}
\usepackage{graphicx}
\usepackage{float}
\usepackage{amsmath}

\counterwithin{figure}{section}

\newtheorem{theorem}{Theorem}[section]
\newtheorem{prop}[theorem]{Proposition}
\newtheorem{lemma}[theorem]{Lemma}

\newtheorem{claim}[theorem]{Claim}

\theoremstyle{definition}

\newtheorem*{defn-non}{Definition}

\newlist{Case}{enumerate}{3}
\setlist[Case, 1]{%
    label           =   {\bfseries Case \arabic*.},
    labelindent=1em ,labelwidth=1cm, labelsep*=1em, leftmargin =!
}
\setlist[Case, 2]{%
    label           =   {\bfseries Subcase \arabic{Casei}.\arabic*.},
    labelindent=-1em ,labelwidth=1cm, labelsep*=1em, leftmargin =!
}
\setlist[Case, 3]{%
    label           =   {\bfseries Subsubcase \arabic{Casei}.\arabic{Caseii}.\arabic*.},
    labelindent=-1em ,labelwidth=1cm, labelsep*=1em, leftmargin =!
}

\newenvironment{poc}{\begin{proof}[Proof of claim]}{\end{proof}}

\usepackage{todonotes}

\title{Algebraic approach to stability results for Erd\H{o}s-Ko-Rado theorem}

\author{
Gennian Ge\thanks{School of Mathematical Sciences, Capital Normal University, Beijing, China. Emails: gnge@zju.edu.cn, 3535935416@qq.com. Gennian Ge was supported by the National Key Research and Development Program of China under Grant 2020YFA0712100, the National Natural Science Foundation of China under Grant 12231014, and Beijing Scholars Program.}
\and 
Zixiang Xu\thanks{Extremal Combinatorics and Probability Group (ECOPRO), Institute for Basic Science (IBS), Daejeon, South Korea. Email: zixiangxu@ibs.re.kr. Supported by IBS-R029-C4.}
\and
Xiaochen Zhao\footnotemark[1]
}

\begin{document}
\maketitle

\begin{abstract}
   Celebrated results often unfold like episodes in a long-running series. In the field of extremal set thoery, Erd\H{o}s, Ko, and Rado in 1961 established that any $k$-uniform intersecting family on $[n]$ has a maximum size of $\binom{n-1}{k-1}$, with the unique extremal structure being a star. In 1967, Hilton and Milner followed up with a pivotal result, showing that if such a family is not a star, its size is at most $\binom{n-1}{k-1} - \binom{n-k-1}{k-1} + 1$, and they identified the corresponding extremal structures. In recent years, Han and Kohayakawa, Kostochka and Mubayi, and Huang and Peng have provided the second and third levels of stability results in this line of research. 

In this paper, we provide a unified approach to proving the stability result for the Erd\H{o}s-Ko-Rado theorem at any level. Our framework primarily relies on a robust linear algebra method, which leverages appropriate non-shadows to effectively handle the structural complexities of these intersecting families. 
  
\end{abstract}
\section{Introduction}
\subsection{Overview}
We say that a set system $\mathcal{F} \subseteq 2^{[n]}$ is \emph{intersecting} if $F_1 \cap F_2 \neq \emptyset$ for every pair of sets $F_1, F_2 \in \mathcal{F}$. A fundamental result in extremal set theory, due to Erd\H{o}s, Ko, and Rado~\cite{1961EKR}, determines the maximum size of a $k$-uniform intersecting family $\mathcal{F}$.

\begin{theorem}[Erd\H{o}s-Ko-Rado~\cite{1961EKR}]\label{thm:EKR}
    Let $n,k$ be positive integers with $n\ge 2k$. If $\mathcal{F}\subseteq \binom{[n]}{k}$ is an intersecting family, then we have
\begin{equation*}
  |\mathcal{F}| \le \binom{n-1}{k-1}.   
\end{equation*}
The equality holds if and only if $\mathcal{F}=\{F\in\binom{[n]}{k}:p\in F\}$ for some $p\in [n]$.
\end{theorem}
Note that the only extremal family in the Erd\H{o}s-Ko-Rado theorem is the \emph{star}, also known as a \emph{trivially intersecting} family. Hilton and Milner~\cite{1967Hilton} established a stability result, showing that the size of non-trivially intersecting families is significantly smaller than the maximum given by the Erd\H{o}s-Ko-Rado theorem. Before presenting a series of stability results, we first introduce some necessary notations. For a set system $\mathcal{F}\subseteq 2^{[n]}$, we define the \emph{maximum degree} of $\mathcal{F}$ as
\begin{equation*}
    d_{\textup{max}}(\mathcal{F}):=\max\limits_{i\in [n]}|\{F\in\mathcal{F}:i\in F\}|.
\end{equation*}
 We use $[i,j]$ to denote the set $\{i,i+1,\ldots,j\}$.
Let $k\ge 3$ be a positive integer, we will use $\binom{[n]}{k}$ to denote the family of all subsets of $[n]$ with size $k$. For $i\in [3,k+1]$, we define
\begin{equation*}
    \mathcal{M}_{i}:=\bigg\{F\in\binom{[n]}{k}:1\in F, F\cap [2,i]\neq \emptyset \bigg\}\cup \bigg\{F\in\binom{[n]}{k}:1\notin F,[2,i]\subseteq F\bigg\}.
\end{equation*}
It is not hard to check that $\mathcal{M}_{i}$ is intersecting and $|\mathcal{M}_{i}|=\sum\limits_{j=2}^{i}\binom{n-j}{k-2}+\binom{n-i}{k-i+1}$. In particular, we can see that $|\mathcal{M}_{3}|=\binom{n-2}{k-2}+2\binom{n-3}{k-2}=\binom{n-2}{k-2}+\binom{n-3}{k-2}+\binom{n-4}{k-2}+\binom{n-4}{k-3}=|\mathcal{M}_{4}|$. Moreover, we can see $d_{\textup{max}}(\mathcal{M}_{i})=\sum\limits_{j=2}^{i}\binom{n-j}{k-2}$. For $j\in [1,n-k]$ we further define 
\begin{align*}
    \mathcal{M}_{k,j}:=\bigg\{F\in\binom{[n]}{k}:1\in F,F\cap [2,k]\neq\emptyset\bigg\} \cup \bigg\{F\in\binom{[n]}{k}:1\in F,F\cap [2,k]=\emptyset,[k+1,k+j]\subseteq F\bigg\}&\\\cup \bigg\{F\in\binom{[n]}{k}:1\notin F,[2,k]\subseteq F, |F\cap [k+1,k+j]|=1\bigg\}.
\end{align*}
In particular, one can easily check that $\mathcal{M}_{k,n-k}=\mathcal{M}_{k},\mathcal{M}_{k,1}=\mathcal{M}_{k+1}$.

For two distinct sets $E_1,E_2\in\binom{[n]}{k}$ with $\left|E_1 \cap E_2\right|=k-2$, and an element $x_0 \in[n] \setminus\left(E_1 \cup E_2\right)$, we define
\begin{equation*}
    \mathcal{K}\left(E_1, E_2, x_0\right):=\left\{G \in\binom{[n]}{k}: x_0 \in G, G \cap E_1 \neq \emptyset,G \cap E_2 \neq \emptyset\right\} \cup\left\{E_1, E_2\right\},
\end{equation*}
 and we write $\mathcal{K}_2$ for any family isomorphic to $\mathcal{K}\left(E_1, E_2, x_0\right)$.

\begin{theorem}[Hilton-Milner~\cite{1967Hilton}]\label{thm:HiltonMilner}
Let $n,k$ be positive integers with $n>2k$. If $\mathcal{F}\subseteq \binom{[n]}{k}$ is an intersecting family and $\bigcap_{F\in\mathcal{F}}F=\emptyset$, then we have
\begin{equation*}
  |\mathcal{F}| \le \binom{n-1}{k-1} - \binom{n-k-1}{k-1}+1.   
\end{equation*}
For $k=3$, the equality holds if and only if $\mathcal{F}$ is isomorphic to $\mathcal{M}_{3}$ or $\mathcal{M}_{4}$; for $k\ge 4$, the equality holds if and only if $\mathcal{F}$ is isomoprhic to $\mathcal{M}_{k+1}$.
\end{theorem}
  If the intersecting family $\mathcal{F}\subseteq\binom{[n]}{k}$ is neither extremal Erd\H{o}s-Ko-Rado family nor extremal Hilton-Milner family, what is the maximum size of $\mathcal{F}$? Han and Kohayakawa~\cite{2017PAMSHanJie} resolved this problem. Here we use $\mathcal{F}_{\textup{EKR}}$ and $\mathcal{F}_{\textup{HM}}$ to denote the corresponding extremal families in~\cref{thm:EKR,thm:HiltonMilner} respectively.

 \begin{theorem}[Han-Kohayakawa~\cite{2017PAMSHanJie}]\label{thm:Han}
 Suppose that $k\ge 3$ and $n>2k$, let $\mathcal{F}\subseteq\binom{[n]}{k}$ be an intersecting family. Assume that $\mathcal{F}$ is neither a sub-family of $\mathcal{F}_{\textup{EKR}}$ nor $\mathcal{F}_{\textup{HM}}$, then we have
 \begin{equation*}
     |\mathcal{F}|\le \binom{n-1}{k-1}-\binom{n-k-1}{k-1}-\binom{n-k-2}{k-2}+2.
 \end{equation*}
Moreover, when $k=4$, the equality holds if and only if $\mathcal{F}$ is isomorphic to $\mathcal{M}_{4,2}, \mathcal{M}_3$ or $\mathcal{M}_4$; when $k\ge 5$ or $k=3$, the equality holds if and only if $\mathcal{F}$ is isomorphic to $\mathcal{M}_{k,2}$.
 \end{theorem}

Han and Kohayakawa~\cite{2017PAMSHanJie} further asked what the next maximum intersecting $k$-uniform families on $[n]$ are? Kostochka and Mubayi~\cite{2017PAMSMubayi} answered this question when $n$ is sufficiently large. Recently, Huang and Peng~\cite{2024EUJCPeng} completely answered this question for any $n\ge 2k+1$ as follows.

\begin{theorem}[Huang-Peng~\cite{2024EUJCPeng}]\label{thm:Peng}
Let $k \geq 4$ and $\mathcal{F} \subseteq\binom{[n]}{k}$ be an intersecting family which is neither a sub-family of $\mathcal{F}_{\textup{EKR}}$ nor $\mathcal{F}_{\textup{HM}}$. Furthermore, $\mathcal{F} \not\subseteq \mathcal{M}_{k,2}$, in addition $\mathcal{F} \not\subseteq \mathcal{M}_{3}$ and $\mathcal{F}\not\subseteq \mathcal{M}_4$ if $k=4$.
Then the followings hold.
\begin{enumerate}
    \item[\textup{(1)}] If $2 k+1 \leq n \leq 3 k-3$, then $|\mathcal{F}| \leq\binom{ n-1}{k-1}-2\binom{n-k-1}{k-1}+\binom{n-k-3}{k-1}+2$. Moreover, when $k\ge 5$, the equality holds if and only if $\mathcal{F}$ is isomorphic to $\mathcal{K}_2$. When $k=4$, the equality holds if and only if $\mathcal{F}$ is isomorphic to $\mathcal{K}_2$ or $\mathcal{M}_{4,3}$.
    \item[\textup{(2)}] If $n \geq 3 k-2$, then $|\mathcal{F}| \leq\binom{ n-1}{k-1}-\binom{n-k-1}{k-1}-\binom{n-k-2}{k-2}-\binom{n-k-3}{k-3}+3$. Moreover, when $k=5$, the equality holds if and only if $\mathcal{F}$ is isomorphic to $\mathcal{M}_{5,3}$ or $\mathcal{M}_5$. For every other $k$, the equality holds if and only if $\mathcal{F}$ is isomorphic to $\mathcal{M}_{k,3}$.
\end{enumerate}
\end{theorem}
The theorems above can be viewed as the 1st, 2nd, and 3rd-level full stability results for the Erd\H{o}s-Ko-Rado theorem, respectively. Currently, there are several different proofs for these stability results, particularly regarding the Hilton-Milner theorem~\cite{2019MoscFrankHM,1986JCTAFFHM,1992JCTAFTHM,2018DMHM,2018JCTA,1985GC} and some interesting variants~\cite{2018JCTAFrankl,2019KupavskiiJCTA}. However, as far as we know, no work has emerged that approaches proving these stability results from multilinear polynomial methods. We aim to fill this gap.

The main contribution of this paper is to provide an algebraic approach to proving the stability result at the $t$-th level of the Erd\H{o}s-Ko-Rado theorem for arbitrary positive integer $t$. To more clearly demonstrate our method, we will present a slightly weaker version of the stability result under the following conditions. Specifically, we make two assumptions: first, for the stability result at the $t$-th level, we assume $k \geq t + 2$; second, we assume $n > \frac{(5+\sqrt{5})k-7}{2}\approx 3.618k$ instead of $n > 2k$. 

\begin{theorem}[$t$-th level stability]\label{thm:AllLayerStability}
Let $t\ge 4$ be a positive integer. For any positive integer $k\ge t+2$, let $n>\frac{(5+\sqrt{5})k-7}{2}$ and $\mathcal{F}\subseteq\binom{[n]}{k}$ be a non-trivial intersecting family. Suppose $\mathcal{F}\not\subseteq \mathcal{M}_{k,t_{0}-1}$ for any $2\le t_{0}\le t$, then 
\begin{equation*}
    |\mathcal{F}|\le \binom{n-1}{k-1}-\sum\limits_{j=1}^{t}\binom{n-k-j}{k-j}+t.
\end{equation*}
When $k=t+2$, the equality holds if and only if $\mathcal{F}$ is isomorphic to $\mathcal{M}_{t+2,t}$ or $\mathcal{M}_{t+2}$, and when $k>t+2$, the equality holds if and only if $\mathcal{F}$ is isomorphic to $\mathcal{M}_{k,t}$. 
\end{theorem}

Note that with our framework, combined with some simple and appropriate structural analysis, we can also obtain stability results in~\cref{thm:HiltonMilner,thm:Han,thm:Peng} mentioned earlier, as well as higher-layer stability results. However, since the main purpose of this paper is to present a new framework and method, we will not fully expand on these full proofs. For interested readers, we will provide an alternative proof of~\cref{thm:Han} in the Appendix using our unified framework.

We remark that in~\cref{thm:AllLayerStability}, we provide the upper bound for $|\mathcal{F}|$ under the assumption that we prohibit all of the extremal structures up to the first $t-1$ levels. Moreover, we can characterize the corresponding extremal structures. After completing this draft, we are informed by Jian Wang that recent work by Kupavskii, and by Frankl and Wang~\cite{2021Diversity,2024WangEUJC,2018Kupa} established a similar upper bound under the condition that the \emph{diversity} $\gamma(\mathcal{F})=|\mathcal{F}|-d_{\textup{max}}(\mathcal{F})$ is large. Although we do not use the concept of diversity here, it has proven valuable in the study of $k$-uniform intersecting families. We recommend the interested readers to~\cite{2021Hung,2018CPC,2021Diversity,2018Kupa,2019KupavskiiJCTA} and the references therein. Instead, we will introduce our framework and give the self-contained and elementary proofs in~\cref{section:Framework}. When we finally try to analyze the extremal structures in~\cref{section:ProofS} and~Appendix, for convenience, we will use a celebrated result of Frankl~\cite{1987JCTAFrankl}, see~\cref{thm:Frankl}.

\section{Robust linear algebra methods}
\subsection{Some useful lemmas}
The following triangular criterion is useful when we want to prove a sequence of polynomials to be linearly independent.

\begin{prop}\label{prop:triangular}
    Let $f_{1},f_{2},\ldots,f_{m}$ be functions in a linear space. If $\boldsymbol{v}^{(1)},\boldsymbol{v}^{(2)},\ldots,\boldsymbol{v}^{(m)}$ are vectors such that $f_{i}(\boldsymbol{v}^{(i)})\neq 0$ for $1\le i\le m$ and $f_{i}(\boldsymbol{v}^{(j)})=0$ for $i>j$, then $f_{1},f_{2},\ldots,f_{m}$ are linearly independent.
\end{prop} 
We follow the notations on a proof of Erd\H{o}s-Ko-Rado theorem via multilinear polynomials~\cite{2006PoLyEKR}. Suppose that $\mathcal{F}:=\{F_{1},F_{2},\ldots,F_{m}\}$, where $F_{i}\subseteq [n]$ for each $1\le i\le m$. For a set $P\subseteq [n]$ and a non-negative integer $\beta$, we say the set $F$ satisfies the property \emph{$(P,\beta)$-intersection} if $|F\cap P|=\beta$. Now suppose for each $F_{i}\in\mathcal{F}$, we write a collection of $s$ intersection properties (allow repetition) as
\begin{equation*}
    R_{i}=\{(P_{i_1},\beta_{i_1}),(P_{i_2},\beta_{i_2})),\ldots,(P_{i_{s}},\beta_{i_{s}})\}.
\end{equation*}
In~\cite{2006PoLyEKR}, the authors built the relation between the multilinear polynomials and the certain collections of intersection properties, here we introduce the following key lemma and the proof in details.

 \begin{lemma}\label{HowToLP}
    Let $\mathcal{F}=\{F_{1},F_{2},\ldots,F_{m}\}\subseteq 2^{[n]}$. Suppose that for each $F_{i}\in \mathcal{F}$, one can find a set $X_{i}\subseteq [n]$ and a collection of $s$ intersection properties $R_{i}$ such that 
     \begin{enumerate}
         \item[\textup{(1)}] $X_{i}$ does not satisfy any of the conditions in $R_{i}$;
         \item[\textup{(2)}] $X_{i}$ satisfies at least one condition in $R_{j}$ for all $j>i$.
     \end{enumerate}
     Then we have $|\mathcal{F}|\le\sum\limits_{h=0}^{s}\binom{n}{h}$.
 \end{lemma}
 \begin{proof}[Proof of~\cref{HowToLP}]
    For $\boldsymbol{x}=(x_{1},x_{2},\ldots,x_{n})$, we define a sequence of $n$-variate real polynomials $f_{i}(\boldsymbol{x})$ for $1\le i\le m$ as
     \begin{equation*}
         f_{i}(\boldsymbol{x})=\prod\limits_{a=1}^{s}\bigg(\sum\limits_{b\in P_{i_{a}}}x_{b}-\beta_{i_{a}}\bigg).
     \end{equation*}
 For a subset $A\subseteq [n]$, we will use $\boldsymbol{a}=(a_{1},a_{2},\ldots,a_{n})$ to represent its characteristic vector, that is, for each $1\le i\le n$, $a_{i}=1$ if $i\in A$ and $a_{i}=0$ otherwise. Observe that the scale product $\boldsymbol{a}\cdot\boldsymbol{b}=|A\cap B|$ for any $A,B\subseteq [n]$. Thus, we can write $f_{i}(\boldsymbol{x})$ as
 \begin{equation*}
     f_{i}(\boldsymbol{x})=\prod\limits_{a=1}^{s}(|X\cap P_{i_{a}}|-\beta_{i_{a}}).
 \end{equation*}
 Let $\boldsymbol{x}^{(i)}$ be the characteristic vector of set $X_{i}$. By condition~(1), we can see $f_{i}(\boldsymbol{x}^{(i)})\neq 0$. By condition~(2), $f_{j}(\boldsymbol{x}^{(i)})=0$ for all $j>i$. Then by Proposition~\ref{prop:triangular}, $\{f_{i}\}_{i=1}^{m}$ are linearly independent. Moreover, as each polynomial contains $n$ variables and the degree of each polynomial is at most $s$, thus we have $m\le\sum\limits_{h=0}^{s}\binom{n}{h}$, as claimed.
 \end{proof}

When we analyze the structural properties of set systems in proofs of~\cref{thm:Han} and~\cref{thm:AllLayerStability}, we will take advantage of the following result of Frankl~\cite{1987JCTAFrankl}.

\begin{theorem}\label{thm:Frankl}
Suppose that $n>2k$, $3\le i\le k+1$, $\mathcal{F}\subseteq\binom{[n]}{k}$ is an intersecting family with $d_{\textup{max}}(\mathcal{F})\le d_{\textup{max}}(\mathcal{M}_{i})$, then $|\mathcal{F}|\le|\mathcal{M}_{i}|$. Moreover if $|\mathcal{F}|=|\mathcal{M}_{i}|$, then either $\mathcal{F}$ is isomorphic to $\mathcal{M}_{i}$, or when $i=4$, $\mathcal{F}$ is isomorphic to $\mathcal{M}_{3}$.
    
\end{theorem}
\subsection{Overview of the robust linear algebra methods}
Our main approach is based on a robust linear algebra method developed in recent work of Gao, Liu and the second author~\cite{2023StabilityComb}.
Here we first show an example and explain how the standard linear algebra method works and then summarize some interesting tricks. Furthermore, we will briefly introduce the main ideas on the robust linear algebra method. For more on the linear algebra methods in combinatorics, we recommend the interested readers to the great textbook~\cite{2020FranklBabai} and a recent note~\cite{2021Lisa}.

Suppose that $\mathcal{F}\subseteq 2^{[n]}$ is an $L$-intersecting family for some subset $L\subseteq [n]$ with $|L|=s$, Frankl and Wilson~\cite{1981FranklWilson} showed that $|\mathcal{F}|\leq \sum\limits_{i=0}^{s}\binom{n}{i}$. We sketch the shorter proof of the above result by Babai~\cite{1988Babai} as follows. Let $\mathcal{F}=\{F_{1},F_{2},\ldots,F_{m}\}$, indexed so that $|F_{1}|\leq\cdots\leq |F_{m}|$ and $L=\{\ell_{1},\ldots,\ell_{s}\}$. For each $i$ let $\vec v_{i}$ be the incidence vector of $F_{i}$. Define polynomials $f_{1},f_{2},\ldots,f_{m}$ by $f_{i}(\vec x)=\prod\limits_{\ell_{k}<|F_{i}|}(\vec{x}\cdot\vec v_{i}-\ell_{k})$. Then one can easily prove that $\{f_{i}\}_{i=1}^{m}$ are linearly independent via Proposition~\ref{prop:triangular}. Thus $|\mathcal{F}|$ can be upper bounded by the number of all possible polynomials. The first trick one can use is multilinear reduction, that is, we can replace each $x_{i}^{t}$ term by $x_{i}$ for any $1\leq i\leq n$ and $t\geq 2$ because $x_{i}\in\{0,1\}$. Using this trick, one can efficiently give the upper bound $\sum\limits_{i=0}^{s}\binom{n}{i}$ for $|\mathcal{F}|$ as the total degree of each monomial is at most $s$.

The second trick is that, one can further add more associated polynomials. For example, we can add $\sum_{i=0}^{s-1}\binom{n-1}{i}$ many extra polynomials in the following way. Label the sets in $\binom{[n-1]}{\leq s}$ with label $B_{i}$ for $1\leq i\leq q= \sum_{i=0}^{s-1}\binom{n-1}{i}$ such that $|B_{i}|\leq |B_{j}|$ when $i<j$. Let $\vec w_{i}$ be the characteristic vector of $B_{i}$, and let $h_{B_{i}}(\vec x)=\prod_{j\in B_{i}}x_{j}$ for $i>1$. Then define a multilinear polynomial $g_{B_{i}}$ in $n$ variables as follows. $g_{B_{1}}=x_{n}-1$ and $g_{B_{i}}=(x_{n}-1)h_{B_{i}}(\vec x)$ for $i>1$. In~\cite{2003Snevily}, Snevily proved that $\{f_{i}\}_{i=1}^{m}$ and $\{g_{B_{j}}\}_{j=1}^{q}$ are linearly independent. Thus the upper bound for $|\mathcal{F}|$ can be improved to $\sum_{i=0}^{s}\binom{n}{i}-\sum_{i=0}^{s-1}\binom{n-1}{i}=\sum_{i=0}^{s}\binom{n-1}{i}$. This trick has been applied to several problems, for example, see~\cite{ 1991AlonBabaiSuzuki, 1983Bannai, 1984Blokhuis, 2009JCTAChen, 2007EUJC, 2007JACMubayi, 1975Ray}. In a word, this trick consists of two parts, the first step is to choose some appropriate extra polynomials, one can see that in many previous works involved with this trick, the extra polynomials usually are clear and natural, which associate some explicit family of subsets. For instance in the above famous case, the extra polynomials associate to $\binom{[n-1]}{\leq s}$ exactly. The second part is that we need to show the union of the original polynomials and the extra polynomials are linearly independent. Usually the second part is much more complicated and difficult. Once we prove the linear independence, we can immediately see the improvement on the bounds for size of the family.

In the proofs of stability results of Kleitman's isodiametric inequality, the authors in~\cite{2023StabilityComb} mainly focus on another direction about the second trick. More precisely, one can first carefully choose a family of subsets which satisfies some appropriate properties and then associate each subset of the chosen family with the extra polynomial one by one. Then it will be easier to prove the linear independence. At the cost, one cannot know all of information about the chosen family, but sometimes one can ignore it, e.g., see the proof in~\cite[Theorem~1.10]{2023StabilityComb}. While in some cases, then the main task in the robust linear algebra method is to dig out the structural properties of the family one chooses, to achieve this, usually one can apply some structural analysis, e.g., see the proof in~\cite[Theorem~1.8]{2023StabilityComb}.

There are several advantages in this robust linear algebra method. The first is that the linear independence usually will be easier to show. The second is that, usually the previous linear algebra methods just provide the bound of the size, while the new method can be used to prove some stability results, that means, we can not only capture the size of the family, but also obtain some structural information. We believe that this method has the potential to be applied to a wider range of problems.

\section{A unified framework and proof of~\cref{thm:AllLayerStability}}
\subsection{High level overview of our framework}
Although our proofs are relatively simple, it might be helpful to briefly outline the main ideas.
\begin{enumerate}
    \item Following the ideas in~\cite{2006PoLyEKR}, we partition the family $\mathcal{F}$ into two sub-families, $\mathcal{F}_{0}$ and $\mathcal{F}_{1}$, where the sets in $\mathcal{F}_{1}$ contain the element that attains the maximum degree in $\mathcal{F}$. We then introduce two auxiliary families, $\mathcal{H}$ and $\mathcal{G}$. At this stage, we can associate these families with certain intersection conditions, which leads to an algebraic proof of~\cite{2006PoLyEKR}.
    \item The key ingredient involves finding one more auxiliary family $\mathcal{S}$, which is a sub-family of the non-shadows of $\mathcal{F}_{1}$. We carefully associate $\mathcal{S}$ with appropriate intersection conditions. The first crucial point occurs at~\cref{claim:LID}, after proving that the families $\mathcal{F}_{1}$, $\mathcal{G}$, $\mathcal{H}$, and $\mathcal{S}$ satisfy certain properties analogous to linear independence in normal linear algebra method, we obtain an important quantitative relation: $|\mathcal{F}_{1}| \leq \binom{n-1}{k-1} - |\mathcal{S}|$, leading to the inequality $|\mathcal{F}| \leq \binom{n-1}{k-1} - |\mathcal{S}| + |\mathcal{F}_{0}|$.
    \item To show the upper bound on $|\mathcal{F}|$, it is sufficient to analyze the lower bound on $|\mathcal{S}| - |\mathcal{F}_{0}|$. We establish a general lower bound on $|\mathcal{S}|$ in~\cref{claim:SizeOfS}, and later, through a stability argument, refine this bound in~\cref{claim:StabilityForS}, which is essential for determining the extremal structures.
    \item The final task is to carefully compare the specific sizes of several numbers. To do this, we establish inequalities between variables, explore monotonicity, and analyze the structure at extreme values in~\cref{claim:SuanTMD} and~\cref{claim:MonotoneInT}. Additionally, we will use~\cref{thm:Frankl} of Frankl. These steps are relatively straightforward and follow from natural analytical considerations.
\end{enumerate}

\subsection{Our framework}\label{section:Framework}
For positive integers $n\ge 2k+1$, without loss of generality, we can assume that $\mathcal{F}$ is a maximal non-trivial intersecting family of $\binom{[n]}{k}$, that is, if we add any new member of $\binom{[n]}{k}$ to $\mathcal{F}$, then $\mathcal{F}$ is not non-trivial intersecting. For a family $\mathcal{F}\in\binom{[n]}{k}$, we denote the \emph{shadow} of $\mathcal{F}$ as $\partial_{k-1}{\mathcal{F}}:=\{T\in\binom{[n]}{k-1}:T\subseteq F\ \text{for\ some\ }F\in\mathcal{F}\}$.
  Let $p\in [n]$ be the element which attains the maximum degree in $\mathcal{F}$.
Consider the following families:
\begin{itemize}
    \item $\mathcal{F}_{0}:=\{F\in\mathcal{F}:p\notin F\}$;
    \item $\mathcal{H}:=\{H\subseteq [n]: p\notin H, 0\le |H|\le k-2\}$;
    \item $\mathcal{F}_{1}:=\{F\in\mathcal{F}: p\in F\}$;
    \item $\mathcal{G}:=\{G\subseteq [n]:p\in G, 1\le |G|\le k-1\}$;
    \item $\mathcal{S}:=\{S\subseteq \binom{[n]}{k-1}\setminus\partial_{k-1}{\mathcal{F}_{1}}:p\notin S,  \exists   F\in\mathcal{F}_{0} \textup{ such\ that\ }
 S\cap F =  \emptyset \}$.
\end{itemize}
Let $\mathcal{A}:=\mathcal{F}_{1}\sqcup\mathcal{H}\sqcup\mathcal{G}\sqcup\mathcal{S}=\{A_{1},\ldots,A_{m}\}$. We define an ordering $\prec$ on the sets, and for two families $\mathcal{A},\mathcal{B}$, denote $\mathcal{A} \prec \mathcal{B}$ if and only if for any $A \in \mathcal{A}$ and $B \in \mathcal{B}$, we have $A \prec B$. We first arrange the sets in a linear order as follows: $\mathcal{H}\prec \mathcal{F}_{1}\prec \mathcal{G}\prec\mathcal{S}$. We put the members of $\mathcal{F}_{1}$ and $\mathcal{S}$ in arbitrary order and the members of $\mathcal{H}$ and $\mathcal{G}$ in order increasing by size, for example, $H_{i}\prec H_{j}$ if $|H_{i}|\le |H_{j}|$. To apply Lemma~\ref{HowToLP}, we need to associate each member $A\in\mathcal{A}$ with a set $X$ and at most $k-1$ many intersection conditions as follows.

\begin{itemize}
    \item For $H\in\mathcal{H}$, we can set $X:=H$ with intersection conditions $(\{h\},0)$ for each $h\in H$ and $([n],n-k-1)$.
    \item For $F\in\mathcal{F}_{1}$, we can set $X:=F\setminus\{p\}$ with intersection conditions $(F\setminus\{p\},\beta)$ for $0\le \beta\le k-2$.
    \item For $G\in\mathcal{G}$, we can set $X:=G$ with intersection conditions $(\{g\},0)$ for each $g\in G$.
    \item For $S\in\mathcal{S}$, we can set $X:=S$ with intersection conditions $(\{s\},0)$ for each $s\in S$.
\end{itemize}

We claim that the system $(A_{i},X_{i},R_{i})$, in which $A_{i}\in \mathcal{A}$, $X_{i}$ and the intersections conditions $R_{i}$ are defined as above, satisfies the conditions in Lemma~\ref{HowToLP} with $s=k-1$.

\begin{claim}\label{claim:LID}
    For each $A_{i}\in\mathcal{A}$, $X_{i}$ does not satisfy any of the conditions in $R_{i}$, and $X_{i}$ satisfies at least one condition in $R_{j}$ for all $j>i$. In particular, $|\mathcal{A}|\le\sum\limits_{\ell=1}^{k-1}\binom{n}{\ell}$.
\end{claim}
\begin{poc}
Recall that $\mathcal{H}\prec\mathcal{F}_{1}\prec\mathcal{G}\prec\mathcal{S}$, and we put the elements of $\mathcal{F}_{1}$ and $\mathcal{S}$ in arbitrary order and the members of $\mathcal{H}$ and $\mathcal{G}$ in order increasing by size. For distinct sets $A_{i}\preceq A_{j}\in \mathcal{A}$, we write $A_{i}\rightarrow A_{j}$ if $X_{i}$ satisfies at least one condition in $R_{j}$, and in particular, $A_{i}\rightarrow A_{i}$ if $X_{i}$ does not satisfy any of the conditions in $R_{i}$. For sub-families $\mathcal{X}\preceq\mathcal{Y}$ of $\mathcal{A}$, we write $\mathcal{X}\rightarrow\mathcal{Y}$ if for every $X\in \mathcal{X}$ and $Y\in\mathcal{Y}$ with $X\preceq Y$, we have $X\rightarrow Y$. Then it suffices to check the following $10$ situations: 

\begin{enumerate}
\item $\mathcal{H}\rightarrow\mathcal{H}$: For each $H_{i}\in\mathcal{H}$, obviously it does not satisfy $(\{h\},0)$ for each $h\in H_{i}$ and $([n],n-k-1)$ since $|H|\le k-2$, moreover, for any other $H_{j}$ with $|H_{j}|\ge |H_{i}|$, there exists some $h\in H_{j}$ such that $h\notin H_{i}$. Then $\mathcal{H}\rightarrow\mathcal{H}$ is checked.

 \item $\mathcal{H}\rightarrow \mathcal{F}_{1}$: For any $H\in\mathcal{H}$ and any $F\in\mathcal{F}_{1}$, we have $0\le |H\cap F|\le k-2$, then $\mathcal{H}\rightarrow\mathcal{F}_{1}$ is checked. 
    \item $\mathcal{H}\rightarrow\mathcal{G}$: Since for any $H\in\mathcal{H}$ and $G\in\mathcal{G}$, we have $p\notin H$ and $p\in G$, then $\mathcal{H}\rightarrow\mathcal{G}$ is checked. 
    
    \item $\mathcal{H}\rightarrow\mathcal{S}$: 
    Since for any $H\in\mathcal{H}$ and $S\in\mathcal{S}$, $|S|>|H|$, there exists some $s\in S$ such that $s\notin H$, then $\mathcal{H}\rightarrow\mathcal{S}$ is checked.

      \item $\mathcal{F}_{1}\rightarrow\mathcal{F}_{1}$: Since $\mathcal{F}_{1}$ is $k$-uniform, for any distinct $F_{i}\prec F_{j}\in\mathcal{F}_{1}$, we have $|F_{i}\cap F_{j}|\le k-1$ and $p\in F_{i}\cap F_{j}$. Therefore there exists some $x\in F_{j}\setminus\{p\}$ and $x\notin F_{i}\setminus\{p\}$, then $\mathcal{F}_{1}\rightarrow\mathcal{F}_{1}$ is checked.
        \item $\mathcal{F}_{1}\rightarrow \mathcal{G}$: Since for any $F\in \mathcal{F}_{1}$ and $G\in \mathcal{G}$, we can see $p\in G$ and $p\notin F\setminus\{p\}$, then $\mathcal{F}_{1}\rightarrow \mathcal{G}$ is checked. 
        \item $\mathcal{F}_{1}\rightarrow\mathcal{S}$: Since for any $F\in\mathcal{F}_{1}$ and $S\in\mathcal{S}$, $S\subseteq\binom{[n]}{k-1}\setminus\partial_{k-1}\mathcal{F}_{1}$, then there exists some $x\in S$ such that $x\notin F\setminus\{p\}$. Then $\mathcal{F}_{1}\rightarrow\mathcal{S}$ is checked.

         \item $\mathcal{G}\rightarrow\mathcal{G}$: For any distinct $G_{i}\prec G_{j}\in\mathcal{G}$ with $|G_{i}|\le|G_{j}|$, obviously there exists some $g\in G_{j}$ such that $G\notin G_{i}$, then $\mathcal{G}\rightarrow\mathcal{G}$ is checked.
      \item $\mathcal{G}\rightarrow\mathcal{S}$: Since for any $G\in\mathcal{G}$ and $S\in\mathcal{S}$, we have $p\in G$ and $p\notin S$. Moreover, notice that $|G|\le k-1$ and $|S|=k-1$, therefore there exists some element $s\in S$ such that $s\notin G$. Then $\mathcal{G}\rightarrow\mathcal{S}$ is checked.

      \item $\mathcal{S}\rightarrow\mathcal{S}$: Since for any distinct $S_{i}\prec S_{j}\in\mathcal{S}\subseteq\binom{[n]}{k-1}$ with $S_{i}\prec S_{j}$, there exists some $s\in S_{j}$ such that  $s\notin S_{i}$, then $\mathcal{S}\rightarrow\mathcal{S}$ is checked.
\end{enumerate}
This finishes the proof.
\end{poc}
Then by~\cref{HowToLP} and~\cref{claim:LID}, we have $|\mathcal{H}|+|\mathcal{F}_{1}|+|\mathcal{G}|+|\mathcal{S}|\le\sum\limits_{i=0}^{k-1}\binom{n}{i}$. Also note that $|\mathcal{G}|=|\mathcal{H}|=\sum\limits_{i=0}^{k-2}\binom{n-1}{i}$, which implies
\begin{equation*}
  d_{\textup{max}}(\mathcal{F}) = |\mathcal{F}_{1}|\le \sum\limits_{i=0}^{k-1}\binom{n}{i}-2\sum\limits_{i=0}^{k-2}\binom{n-1}{i}-|\mathcal{S}|=\binom{n-1}{k-1}-|\mathcal{S}|.
\end{equation*}
Moreover, we have
\begin{equation*}
|\mathcal{F}|=|\mathcal{F}_{0}|+|\mathcal{F}_{1}|\le\binom{n-1}{k-1}-(|\mathcal{S}|-|\mathcal{F}_{0}|).
\end{equation*}

We can immediately derive the following results when $|\mathcal{F}_{0}|$ is very small.
\begin{enumerate}
    \item When $|\mathcal{F}_{0}|=0$, suppose that there exists some $S\in\mathcal{S}$, then $S\cup\{p\}\notin \mathcal{F}$ and $(S\cup\{p\})\cap F\neq\emptyset$ for any $F\in\mathcal{F}$, which is a contradiction to that $\mathcal{F}$ is a maximal non-trivial intersecting family. Therefore, when $\mathcal{F}_{0}=\emptyset$, then $\mathcal{S}=\emptyset$, which gives $|\mathcal{F}|\le\binom{n-1}{k-1}$.
    \item When $|\mathcal{F}_{0}|=1$, set $\mathcal{F}_{0}=\{F_{0}\}$. By definitions, for each $S\in\mathcal{S}$, we can see $S\cap F_{0}=\emptyset$ and $p\notin S$. Moreover, we claim that $\binom{[n]\setminus (\{p\}\cup F_{0})}{k-1}\subseteq \mathcal{S}$. To see this, since $\mathcal{F}$ is an intersecting family, then for any member $P\in\partial_{k-1}\mathcal{F}_{1}$, at least one of the events $p\in P$ and $P\cap F_{0}\neq\emptyset$ occurs, which yields that $\partial_{k-1}\mathcal{F}_{1}\cap \binom{[n]\setminus (F_{0}\cup\{p\})}{k-1}=\emptyset$. Therefore, when $|\mathcal{F}_{0}|=1$, then $|\mathcal{S}|\ge \binom{n-k-1}{k-1}$, which also yields that $|\mathcal{F}|\le\binom{n-1}{k-1}-\binom{n-k-1}{k-1}+1$.
\end{enumerate} 
In general, Let $x:=|\mathcal{F}_{0}|\ge 1$, and $\mathcal{F}_{0}:=\{F_{1},F_{2},\ldots,F_{x}\}$, we have the following crucial claim by extending the above argument in the case of $|\mathcal{F}_{0}|=1$.

\begin{claim}\label{claim:LowerBoundForS}
$\mathcal{S}=\bigcup\limits_{i=1}^{x}\binom{[n]\setminus(F_{i}\cup\{p\})}{k-1}$.
\end{claim}
\begin{poc}
    Set $\mathcal{T}:=\bigcup\limits_{i=1}^{x}\binom{[n]\setminus(F_{i}\cup\{p\})}{k-1}$, we first prove that $\mathcal{T}\subseteq \mathcal{S}$. Note that for any $T\in\mathcal{T}$, we can see $p\notin T$ and by definitions there exists some $F_{i}\in\mathcal{F}_{0}$ such that $T\cap F_{i}=\emptyset$. Then it suffices to show that $T$ cannot be a shadow of $\mathcal{F}_{1}$. Suppose that there is a $T\subseteq \mathcal{T}$ such that $T\subseteq F$ for some $F\in\mathcal{F}_{1}$, then $T=F\setminus\{p\}$. Since $\mathcal{F}$ is an intersecting family, we have $(T\cup\{p\})\cap F_{i}\neq\emptyset$ for each $i\in [x]$. However, this is impossible, because $p\notin F_{i}$, and there exists some $F_{i}\in\mathcal{F}_{0}$ such that $T\cap F_{i}=\emptyset$, a contradiction to the definition of $\mathcal{T}$. 
    
    On the other hand, for any $S\in\mathcal{S}$, there exists some $F_{i}\in\mathcal{F}_{0}$ such that $S\cap F_{i}=\emptyset$. Then $S\in\binom{[n]\setminus (F_{i}\cup\{p\})}{k-1}$, which yields $\mathcal{S}\subseteq \mathcal{T}$. This finishes the proof.
\end{poc}
For $\mathcal{F}_{0}=\{F_{1},F_{2},\ldots,F_{x}\}$, let $\mathcal{C}_{1}=\{C\subseteq\binom{[n]}{k-1}:p\notin C,C\cap F_{1}=\emptyset\}$. Moreover, for each $2\le i\le x$, we define
\begin{equation*}
    \mathcal{C}_{i}:=\bigg\{C\subseteq\binom{[n]}{k-1}:p\notin C, C\cap F_{i}=\emptyset, C\cap F_{j}\neq\emptyset\ \textup{for\ any\ }1\le j\le i-1  \bigg\}.
\end{equation*}
By definitions one can see that $\mathcal{C}_{1},\mathcal{C}_{2},\ldots,\mathcal{C}_{x}$ are pairwise disjoint, moreover, it is not hard to see that $\bigcup\limits_{i=1}^{x}\mathcal{C}_{i}=\bigcup\limits_{i=1}^{x}\binom{[n]\setminus(F_{i}\cup\{p\})}{k-1}$.
When $1\le i\le x\le k$, by definition we have $|\mathcal{C}_{i}|\ge\binom{n-k-i}{k-i}$, therefore, we have the following lower bound on $|\mathcal{S}|$ by~\cref{claim:LowerBoundForS}.

\begin{claim}\label{claim:SizeOfS}
    When $1\le x\le k$, $|\mathcal{S}|\ge \sum\limits_{j=1}^{x}\binom{n-k-j}{k-j}$.
\end{claim}

By~\cref{claim:LowerBoundForS} and~\cref{claim:SizeOfS}, we can see $|\mathcal{S}|\ge \sum\limits_{j=1}^{k}\binom{n-k-j}{k-j}$ when $x>k$. Indeed, we can strengthen~\cref{claim:SizeOfS} in the following form, which plays a key role when we analyze the extremal structures. Recall that a family $\mathcal{W}$ is called a \emph{sunflower} with $s$ common elements, if there is a set $S\subseteq [n]$ with $|S|=s$ such that for any distinct $W_{i},W_{j}\in\mathcal{W}$, we have $W_{i}\cap W_{j}=S$.

\begin{claim}\label{claim:StabilityForS}
The followings hold.
\begin{enumerate}
    \item[$\textup{(1)}$]  When $1\le x\le k$, $|\mathcal{S}|= \sum\limits_{j=1}^{x}\binom{n-k-j}{k-j}$ if and only if $\mathcal{F}_{0}=\{F_{1},F_{2},\ldots,F_{x}\}$ forms a sunflower with $k-1$ common elements. In particular, if $\mathcal{F}_{0}$ is not a sunflower with $k-1$ common elements, then $|\mathcal{S}|\ge \sum\limits_{j=1}^{x}\binom{n-k-j}{k-j}+\binom{n-k-3}{k-2}$.
    \item[$\textup{(2)}$]  When $k+1\le x\le n-k$, $|\mathcal{S}|= \sum\limits_{j=1}^{k}\binom{n-k-j}{k-j}$ if and only if $\mathcal{F}_{0}=\{F_{1},F_{2},\ldots,F_{x}\}$ forms a sunflower with $k-1$ common elements. In particular, if $\mathcal{F}_{0}$ is not a sunflower with $k-1$ common elements, then $|\mathcal{S}|\ge \sum\limits_{j=1}^{k}\binom{n-k-j}{k-j}+\binom{n-k-3}{k-2}$.
\end{enumerate}
   
\end{claim}
\begin{poc}
We focus on the case when $1\le x\le k$ and one can apply the almost identical argument to show the remaining cases. On one hand, if for any distinct $i,j\in[x]$, $F_{i}\cap F_{j}=A$ for some $A\subseteq \binom{[n]}{k-1}$, then for each $i\in [x]$, set $F_{i}=A\cup\{a_{i}\}$, it is easy to calculate $|\mathcal{C}_{\ell}|=\binom{n-k-\ell}{k-\ell}$ for each $\ell\in [x]$. This implies that in this case, $|\mathcal{S}|= \sum\limits_{j=1}^{x}\binom{n-k-j}{k-j}$.

On the other hand, we need to show that if $|\mathcal{S}|= \sum\limits_{j=1}^{x}\binom{n-k-j}{k-j}$, then $\mathcal{F}_{0}$ has to be a sunflower with $k-1$ common elements. First suppose that there exists a pair of sets $F_{i},F_{j}\in\mathcal{F}_{0}$ such that $|F_{i}\cap F_{j}|\le k-2$, by suitable re-labelling, we can assume $|F_{1}\cap F_{2}|\le k-2$. Then there exist two elements in $F_{1}\setminus F_{2}$, denoted by $\{a_{1},a_{2}\}\subseteq F_{1}\setminus F_{2}$. Then we can see the size of $\mathcal{C}_{2}$ is at least $\binom{n-k-2}{k-2}+\binom{n-k-3}{k-2}$, since one can pick $\binom{n-k-2}{k-2}$ sets that contain $a_{1}$, and at least $\binom{n-k-3}{k-2}$ sets that do not contain $a_{1}$. Therefore if $|\mathcal{S}|= \sum\limits_{j=1}^{x}\binom{n-k-j}{k-j}$ for any distinct $i,j\in [x]$, we can assume that $|F_{i}\cap F_{j}|=k-1$. When $x=2$, then we are already done. When $x\ge 3$, we set $B:=F_{1}\cap F_{2}$, $F_{1}=B\cup\{b_{1}\}$ and $F_{2}=B\cup\{b_{2}\}$. Then it suffices to show for any distinct $\ell,m\in [x]$, $F_{\ell}\cap F_{m}=B$.
    \begin{Case}
     \item If $\{\ell,m\}\cap\{1,2\}\neq\emptyset$, then by symmetry and suitable re-labelling, it suffices to consider the case that $\ell=1$ and $m=3\in [x]\setminus\{1,2\}$. Suppose that $F_{1}\cap F_{3}\neq B$, then there is some $b\in B$ such that $b\notin F_{3}$, which yields that $b\in (F_{1}\setminus F_{3})\cap (F_{2}\setminus F_{3})$. Moreover since $|F_{3}|=k$ and $|F_{3}\cap F_{1}|=|F_{3}\cap F_{2}|=k-1$, then we have $F_{3}=\{b_{1},b_{2}\}\cup B\setminus\{b\}$. Therefore we can see $|\mathcal{C}_{3}|\ge\binom{n-k-2}{k-2}=\binom{n-k-3}{k-2}+\binom{n-k-3}{k-3}>\binom{n-k-3}{k-3}$, which yields $|\mathcal{S}|\ge \sum\limits_{j=1}^{x}\binom{n-k-j}{k-j}+\binom{n-k-3}{k-2}$, a contradiction.
     \item If $\{\ell,m\}\cap\{1,2\}=\emptyset$, by symmetry we can assume that $\ell=3$ and $m=4$. Suppose that $F_{3}\cap F_{4}=D\neq B=F_{1}\cap F_{2}$, set $F_{3}:=D\cup \{d_{3}\}$ and $F_{4}=D\cup \{d_{4}\}$. Since $|F_{1}\cap F_{3}|=k-1$ and $B\neq D$, we have $b_{1}=d_{3}$, similarly since $|F_{1}\cap F_{4}|=k-1$ and $B\neq D$, we have $b_{1}=d_{4}$. However, this implies that $F_{3}=F_{4}$, a contradiction.
    \end{Case}
    This finishes the proof. One can apply the same argument to show the statement in~$\textup{(2)}$, we omit the repeated details.
\end{poc}

Recall that $|\mathcal{S}|\ge \sum\limits_{j=1}^{k}\binom{n-k-j}{k-j}$ when $x>k$.  We then define the function $g(x)$ to be
\[
g(x) =
\begin{cases}
\sum\limits_{j=1}^{x} \binom{n-k-j}{k-j}, & \text{if } 1 \leq x \leq k, \\
\sum\limits_{j=1}^{k} \binom{n-k-j}{k-j}, & \text{if } k+1 \leq x \leq n-k.
\end{cases}
\]
Next we will carefully determine the values at which the function $f(x):=g(x)-x$ attains its minimum. Recall that $|S|-|\mathcal{F}_{0}|\ge f(|\mathcal{F}_{0}|)$, then to obtain the $t$-th level of stability result for Erd\H{o}s-Ko-Rado theorem, we need to understand the exact values of $h(t) = \min \limits_{t \leq x \leq n-k} f(x)$ since we have $|\mathcal{F}|\le\binom{n-1}{k-1}-h(t)$ when $t\le x\le n-k$.

\begin{claim}\label{claim:SuanTMD}
    For given $1\le t\le n-k$, then the followings hold.
    \begin{itemize}
        \item If $k\ge t+3$, then $h(t)=f(t)<\min \limits_{x\neq t} f(x)$.
        \item If $k= t+2$, then $h(t)=f(t)=f(n-k)< \min \limits_{x\notin \{t,n-k\}} f(x)$.
        \item If $ k\le t+1$, then $h(t)=f(n-k)<\min \limits_{x\neq n-k} f(x)$.
    \end{itemize}
\end{claim}
\begin{poc}
 A straightforward calculation shows that $f(x)$ with $x\in [1,n-k]$  is monotonically increasing when $x\in [1,k]$ and monotonically decreasing when $x\in [k,n-k]$. Therefore $h(t)$ is either $f(t)$ or $f(n-k)$. Note that $f(n-k)-f(t)=\binom{n-k-t}{k-(t+1)}-\binom{n-k-t}{1}$, then the claim follows by computing the exact values directly.
\end{poc}

For given positive integer $x$, let $\mathcal{F}(x)$ be the largest intersecting family among all $\mathcal{F}\subseteq\binom{[n]}{k}$ with $|\mathcal{F}_{0}|=x$. Let $d_{\textup{max}}(x)$ represent the maximum degree of $\mathcal{F}(x)\subseteq\binom{[n]}{k}$. Observe that for any $x\in\mathbb{N}$, $|\mathcal{F}(x)|=x+d_{\textup{max}}(x)$. The following observation is crucial.

\begin{claim}\label{claim:MonotoneInT}
    If $x_{1}>x_{2}$, $d_{\textup{max}}(x_{1})\le d_{\textup{max}}(x_{2})$.
\end{claim}
\begin{poc}
    Suppose that $x_{1}>x_{2}$ and $|\mathcal{F}(x_{1})|> |\mathcal{F}(x_{2})|+(x_{1}-x_{2})$, then one can remove $x_{1}-x_{2}$ many sets from $\mathcal{F}(x_{1})$ to obtain an intersecting family $\mathcal{F}$ of size larger than $|\mathcal{F}(x_{2})|$ with $|\mathcal{F}_{0}|=x_{2}$, a contradiction to the definition of $\mathcal{F}(x_{2})$. 
\end{poc}

\subsection{Stability at arbitrary level: Proof of~\cref{thm:AllLayerStability}}\label{section:ProofS}
With the new framework in hand, we then prove the $t$-level stability result for Erd\H{o}s-Ko-Rado theorem. Based on~\cref{thm:HiltonMilner,thm:Han,thm:Peng}, we then consider the case of $t\ge 4$. Let $\mathcal{F}\subseteq\binom{[n]}{k}$ be a non-trivial intersecting family and $\mathcal{F}\not\subseteq\mathcal{M}_{k,t_{0}-1}$ for any $2\le t_{0}\le t$. Let $\mathcal{F}_{0}:=\{F_{1},F_{2},\ldots,F_{x}\}$. Since when $|\mathcal{F}_{0}|\le 1$, $\mathcal{F}$ is a star, or a sub-family of $\mathcal{M}_{k+1}$, it suffices to consider the case of $|\mathcal{F}_{0}|\ge 2$.
   \begin{claim}\label{claim:LargeT}
      $\min\limits_{2\le |\mathcal{F}_{0}|\le t-1}\{|\mathcal{S}|-|\mathcal{F}_{0}|\} >\min\limits_{t\le |\mathcal{F}_{0}|\le n-k}\{|\mathcal{S}|-|\mathcal{F}_{0}|\}$. 
   \end{claim}
   \begin{poc}
    When $x=|\mathcal{F}_{0}|\le t-1$, under the assumption that $\mathcal{F}\not\subseteq\mathcal{M}_{k,t_{0}-1}$ for any $2\le t_{0}\le t$, by the definition of $\mathcal{M}_{k,t_{0}-1}$, we can see $\mathcal{F}_{0}$ is not a sunflower with $k-1$ common elements. Recall that $\mathcal{C}_{1}=\{C\subseteq\binom{[n]}{k-1}:p\notin C,C\cap F_{1}=\emptyset\}$, and for each $2\le i\le |\mathcal{F}_{0}|$,
   \begin{equation*}
    \mathcal{C}_{i}:=\bigg\{C\subseteq\binom{[n]}{k-1}:p\notin C, C\cap F_{i}=\emptyset, C\cap F_{j}\neq\emptyset\ \textup{for\ any\ }1\le j\le i-1  \bigg\}.
\end{equation*}
Under the condition that $\mathcal{F}_{0}$ is not a sunflower with $k-1$ common elements, we then take advantage of the proof of~\cref{claim:StabilityForS}. If $x=2$, then $|\mathcal{C}_{2}|\ge\binom{n-k-2}{k-2}+\binom{n-k-3}{k-2}$, which yields that $|\mathcal{S}|-|\mathcal{F}_{0}|\ge |\mathcal{C}_{1}|+|\mathcal{C}_{2}|-2 \ge \binom{n-k-1}{k-1}+\binom{n-k-2}{k-2}+\binom{n-k-3}{k-2}-2$. If $x\ge 3$, we can see that $|\mathcal{C}_{1}|+|\mathcal{C}_{2}|+|\mathcal{C}_{3}|\ge \binom{n-k-1}{k-1}+\binom{n-k-2}{k-2}+\binom{n-k-3}{k-3}+\binom{n-k-3}{k-2}$. Moreover, since for any $3\le i\le |\mathcal{F}_{0}|$, $|\mathcal{C}_{i}|\ge\binom{n-k-i}{k-i}\ge 1$, we then conclude that
\begin{equation*}
    \min\limits_{2\le | \mathcal{F}_{0}  | \le t-1}\left \{ |\mathcal{S}|- | \mathcal{F}_{0}  |\right \} \ge \binom{n-k-1}{k-1}+\binom{n-k-2}{k-2}+\binom{n-k-3}{k-2}-2. 
\end{equation*}
Since $k\ge t+2$, by~\cref{claim:SuanTMD}, we have $\min\limits_{ t\le |\mathcal{F}_{0}|\le n-k}\{|\mathcal{S}|-|\mathcal{F}_{0}|\}=\sum\limits_{j=1}^{t}\binom{n-k-j}{k-j}-t$. Note that
\begin{align*}
   \binom{n-k-1}{k-1} +& \binom{n-k-2}{k-2} + \binom{n-k-3}{k-2} - 2 
   \quad - \left(\sum\limits_{j=1}^{t}\binom{n-k-j}{k-j} - t\right) \\ & \geq \binom{n-k-3}{k-2} - 2- \left(\binom{n-k-2}{k-3} - t\right),
\end{align*}
where we take advantage of $\sum\limits_{j=3}^{t}\binom{n-k-j}{k-j}\le \sum\limits_{j=3}^{k}\binom{n-k-j}{k-j}=\binom{n-k-2}{k-3}$. Note that when $n>\frac{(5+\sqrt{5})k-7}{2}$, it is easy to check that $\binom{n-k-3}{k-2} - 2- (\binom{n-k-2}{k-3} - t)$ is strictly larger than $0$. This finishes the proof.
   \end{poc}
By~\cref{claim:LargeT}, it then suffices to consider the case when $x=|\mathcal{F}_{0}|\ge t$. If $x\ge n-k+1$, by~\cref{claim:MonotoneInT}, we have $d_{\textup{max}}(x)\le d_{\textup{max}}(n-k)= \binom{n-1}{k-1}-\sum\limits_{j=1}^{k}\binom{n-k-j}{k-j}=d_{\textup{max}}(\mathcal{M}_{k})$, where we take advantage of the formulas $\binom{n-1}{k-1}=\sum_{j=1}^{k-1}  \binom{n-j-1}{k-2} +\binom{n-k}{k-1}$  and  $\binom{n-k}{k-1}=\sum_{j=1}^{k} \binom{n-k-j}{k-j}.$ Therefore by~\cref{thm:Frankl} in this case we have $|\mathcal{F}|\le|\mathcal{M}_{k}|=\sum\limits_{j=2}^{k}\binom{n-j}{k-1}+n-k$. Moreover,~\cref{thm:Frankl} states that the equality holds if and only if $\mathcal{F}$ is isomorphic to $\mathcal{M}_{k}$. However, in this case we assume that $x=|\mathcal{F}_{0}|>n-k$, therefore $\mathcal{F}$ cannot be isomorphic to $\mathcal{M}_{k}$, which yields that $|\mathcal{F}|<|\mathcal{M}_{k}|=\binom{n-1}{k-1}-f(n-k)$.

By definition of $h(t)$, we can see $h(t)\le f(n-k)$, it then remains to consider the case when $t\le x\le n-k$. Recall that in this case we have $|\mathcal{F}|\le\binom{n-1}{k-1}-h(t)$, we then consider the following cases based on~\cref{claim:SuanTMD}.

\begin{Case}
    \item If $k=t+2$, by~\cref{claim:SuanTMD} we have $h(t)=f(t)=f(n-t-2)$. We then determine the extremal structures when $|\mathcal{F}|=\binom{n-1}{k-1}-h(t)$, according to the size of $\mathcal{F}_{0}$.
    \begin{Case}
        \item If $\mathcal{F}_{0}=\{F_{1},F_{2},\ldots,F_{t}\}$, by~\cref{claim:StabilityForS}, $\mathcal{F}_{0}$ is a sunflower with $t+1$ common elements. We denote $A=\bigcap\limits_{i=1}^{t}F_{i}=\{a_{1},a_{2},\ldots,a_{t+1}\}$, and $F_{\ell}\setminus A=\{b_{\ell}\}$ for each $\ell\in [t]$. Then for any $F\in\mathcal{F}_{1}$, if $F\cap A=\emptyset$, then $\{b_{1},b_{2},\ldots,b_{t}\}\subseteq\mathcal{F}$, therefore $\mathcal{F}_{1}=\{G:\binom{[n]}{t+2}:p\in G,G\cap A\neq\emptyset\}\cup\{G\in\binom{[n]}{t+2}:\{b_{1},b_{2},\ldots,b_{t},p\}\in G\}$. Since $\mathcal{F}_{0}$ is a sunflower with $t+1$ common elements, we can see $\mathcal{F}$ is isomorphic to $\mathcal{M}_{t+2,t}$ in this case, as desired.
        \item If $|\mathcal{F}_{0}|=n-t-2$, we can see $d_{\textup{max}}(n-t-2)\le d_{\textup{max}}(\mathcal{M}_{t+2})$, then by~\cref{thm:Frankl}, $|\mathcal{F}|\le|\mathcal{M}_{t+2}|$. Moreover, the equality holds if and only if $\mathcal{F}$ is isomorphic to $\mathcal{M}_{t+2}$, as desired.
    \end{Case}
    \item When $k\ge t+3$, by~\cref{claim:SuanTMD}, we have $h(t)=f(t)=\sum\limits_{j=1}^{t}\binom{n-k-j}{k-j}-t$. Then it suffices to consider the case when $|\mathcal{F}_{0}|=t$. By an almost identical argument as that in Subcase 1.1, we can see $\mathcal{F}$ is isomorphic to $\mathcal{M}_{k,t}$, we omit the details here.
\end{Case}
This finishes the proof. 

\section{Concluding remarks}
In this paper, we focus on developing a unified framework for deriving stability results for the celebrated Erd\H{o}s-Ko-Rado theorem using a robust linear algebraic approach. To illustrate our main ideas, we first present a slightly weaker version that applies to arbitrary levels, extending beyond previous results in~\cite{2017PAMSHanJie,1967Hilton,2024EUJCPeng,2017PAMSMubayi}. In~\cref{thm:AllLayerStability}, we impose two stronger assumptions: $k \ge t + 2$ and $n > \frac{(5+\sqrt{5})k - 7}{2}$. Indeed, by carefully extending the arguments in the proofs of~\cref{claim:StabilityForS},~\cref{claim:SuanTMD} and~\cref{claim:LargeT}, the case under the natural conditions $n \ge 2k + 1$ and $k \ge t + 1$ can be readily analyzed. In fact, we provide an alternative proof of~\cref{thm:Han} in the Appendix, While it is feasible to further extend our methods to give a full proof of~\cref{thm:Peng}, we do not pursue it here. 
It is worth noting that in the Appendix, where we prove~\cref{thm:Han}, the case $k=3=2+1$ is relatively more complicated. Similarly, when proving the complete stability of~\cref{thm:Peng} (i.e., for the third level), we find that the case $k=4=3+1$ is also more intricate. However, interestingly, when $t\ge 4$ the case $k=t+1$ actually becomes simpler. Although we do not fully present this proof in the paper, we encourage interested readers to explore this phenomenon independently.

We believe that a more interesting direction for research is to continue exploring the potential of this robust linear algebra method to obtain more stability results more efficiently.

\section*{Acknowledgement}
Zixiang Xu would like to thank Prof. Hao Huang, Dr. Yongtao Li, Prof. Hong Liu, Prof. Benjian Lv, and Prof. Jian Wang for their valuable discussion during the early stages of this work over the past two years. In particular, he thanks Dr. Yongtao Li for providing the reference~\cite{2006PoLyEKR} and Prof. Jian Wang for informing him of the results in~\cite{2024WangEUJC, 2018Kupa}.

\bibliographystyle{abbrv}
\bibliography{HMProof}

\section*{Appendix: An alternative proof of~\cref{thm:Han}}\label{section:alter}
Since when $|\mathcal{F}_{0}|\le 1$, $\mathcal{F}$ is a star, or a sub-family of $\mathcal{M}_{k+1}$, it suffices to consider the case when $|\mathcal{F}_{0}|\ge 2$. By~\cref{claim:MonotoneInT}, for any $x> n-k$, we have $d_{\textup{max}}(x)\le d_{\textup{max}}(n-k)\le \binom{n-1}{k-1}-\sum_{i=1}^{k}\binom{n-k-i}{k-i}=d_{\textup{max}}(\mathcal{M}_{k})$, which yields $|\mathcal{F}|\le |\mathcal{M}_{k}|=\sum_{j=2}^{k}\binom{n-j}{k-1}+n-k$. Moreover, by~\cref{thm:Frankl} the equality holds if and only if $\mathcal{F}$ is isomorphic to $\mathcal{M}_{k}$, in particular, when $k=4$, it could be isomorphic to $\mathcal{M}_{3}$. Therefore, generally in this case except $k=4$, when $x>n-k$, $\mathcal{F}$ cannot be isomorphic to $\mathcal{M}_{k}$, which implies $|\mathcal{F}|<|\mathcal{M}_{k}|$. Additionally, when $k=4$, if $|\mathcal{F}|=|\mathcal{M}_{4}|$, $\mathcal{F}$ can be isomorphic to $\mathcal{M}_{3}$.

When $2 \le x\le n-k$, note that $|\mathcal{F}|\le\binom{n-1}{k-1}-h(2)$, we consider the following cases.
    \begin{Case}
        \item When $k=3$, by~\cref{claim:SuanTMD}, $h(2)=f(n-3)$, then $|\mathcal{F}|\le\binom{n-1}{2}-f(n-3)\le\binom{n-1}{2}-(\binom{n-4}{2}+\binom{n-5}{1}+1)+(n-3)=\binom{n-1}{2}-\binom{n-4}{2}+1$. Moreover, when $|\mathcal{F}_{0}|=n-3$ and the above equality holds, then by~\cref{claim:StabilityForS} we can see that $\mathcal{F}_{0}$ is a sunflower with $2$ core elements $a_{1},a_{2}$. Since $\mathcal{F}$ is $3$-uniform and $n-3>3$, then any $F\in\mathcal{F}\setminus \mathcal{F}_{0}$ must contain either $a_{1}$ or $a_{2}$, therefore $\mathcal{F}=\{F\in\binom{[n]}{3}:|F\cap\{a_{1},a_{2},p\}|\ge 2\}$. Indeed, this configuration is an extremal structure in~\cref{thm:HiltonMilner}, which is a contradiction to the assumption.
\begin{Case}
    \item When $|\mathcal{F}_{0}|=2$, by~\cref{claim:SizeOfS}, $|\mathcal{F}|\le \binom{n-1}{2}-|\mathcal{S}|+2 \le \binom{n-1}{2} -\binom{n-4}{2}-\binom{n-5}{1}+2=2n-2$. By~\cref{claim:StabilityForS} the equality holds if and only if $\mathcal{F}_{0}$ is a sunflower with $2$ common elements, therefore, $|\mathcal{F}|=2n-2$ if and only if $\mathcal{F}$ is isomorphic to $\mathcal{M}_{3,2}$.
    \item When $3\le|\mathcal{F}_{0}|\le n-3$, note that $\mathcal{F}$ is not a sub-family of extremal structures in~\cref{thm:HiltonMilner}, namely $\mathcal{M}_{4}$ or $\mathcal{M}_{3}$. Suppose that $\mathcal{F}_{0}$ is a sunflower with $2$ common elements, say $\mathcal{F}_{0}=\{F\in\binom{[n]}{3}:\{w_{1},w_{2}\}\subseteq F\}$, then since $\mathcal{F}_{1}$ is a star and $\mathcal{F}$ is intersecting, then $\mathcal{F}$ must be a sub-family of $\mathcal{M}_{3}$, a contradiction to the assumption. Therefore $\mathcal{F}_{0}$ is not a sunflower with $2$ common elements, then by~\cref{claim:StabilityForS}, $|\mathcal{F}|\le\binom{n-1}{2}-|\mathcal{S}|+|\mathcal{F}_{0}|\le\binom{n-1}{2}-\binom{n-4}{2}-\binom{n-5}{1}-\binom{n-6}{1}-1+|\mathcal{F}_{0}|$, in particular, the equality holds only if $|\mathcal{C}_{3}|=1$. Then it suffices to consider the case $|\mathcal{F}_{0}|=n-3$, otherwise $|\mathcal{F}|\le 2n-3$. Since $|\mathcal{F}_{0}|=n-3>3$, there exists some $F_{3}\in\mathcal{F}_{0}$ such that $|F_{1}\setminus F_{3}|=2$ or $|F_{2}\setminus F_{3}|=2$. Suppose that former case occurs, set $F_{1}\setminus F_{3}:=\{a_{1},a_{2}\}$, and set $b\in F_{2}\setminus F_{3}$, then $\{a_{1},b\},\{a_{2},b\}\subseteq \mathcal{C}_{3}$, which implies that $|\mathcal{C}_{3}|\ge 2$, combining with the proof in~\cref{claim:StabilityForS}, we can see $|\mathcal{F}|\le\binom{n-1}{2}-|\mathcal{S}|+n-3\le\binom{n-1}{2}-\binom{n-4}{2}-\binom{n-5}{1}-\binom{n-6}{1}-2+(n-3)= 2n-3$.
    \item When $|\mathcal{F}_{0}|\ge n-2$, we first have the following claim.
    \begin{claim}\label{claim:NOcommon}
        $\bigcap_{F\in\mathcal{F}_{0}}F=\emptyset$.
    \end{claim}
    \begin{poc}
    Suppose that $g\in \bigcap\limits_{F\in\mathcal{F}_{0}}F$, then we consider the sub-family $ \mathcal{F}_{1,g}:=\{F\in\mathcal{F}_{1}:g\notin F\}$ of $\mathcal{F}_{1}$. Pick arbitrary set $A\in \mathcal{F}_{1,g}$, denoted by $A=\{a_{1},a_{2},p\}$. Suppose that there exists some $a_{3}\neq g$ such that $\{a_{1},a_{3},p\}\in\mathcal{F}_{1,g}$, then the subset in $\mathcal{F}_{0}$ that contains $a_{2}$ but does not contain $a_{1}$ must be $\{a_{2},a_{3},g\}$. Then $\mathcal{F}_{0}=\{a_{2},a_{3},g\}\cup \bigg(\bigcup\limits_{a\in [n]\setminus\{p,a_{1},g\}}\{a,a_{1},g\}\bigg)$, which implies $\mathcal{F}_{1,g}$ must be $\{a_{1},a_{3},p\}\cup\{a_{1},a_{2},p\}$. By symmetry, we can conclude that in this case, we have $|\mathcal{F}_{1,g}|\le 2$, which implies that the degree of $g$ is larger than the degree of $p$, a contradiction. If there is no $a_{3}$ such that $\{a_{1},a_{3},p\}\in\mathcal{F}_{1,g}$ or $\{a_{2},a_{3},p\}\in\mathcal{F}_{1,g}$, then either $|\mathcal{F}_{1,g}|\le 1$, or there exist some $b_{1},b_{2}$ with $\{b_{1},b_{2}\}\cap \{a_{1},a_{2}\}=\emptyset$ such that $\{b_{1},b_{2},p\}\in\mathcal{F}_{1,g}$. However, this implies that $\mathcal{F}_{0}\subseteq\{a_{1},b_{1},p\}\cup \{a_{1},b_{2},p\}\cup\{a_{2},b_{1},p\}\cup \{a_{2},b_{2},p\}$, which is a contradiction to $|\mathcal{F}_{0}|\ge n-2>4$. 
 \end{poc}
 Now let $q$ be the element attaining the maximum degree of $\mathcal{F}_{0}$, and denote $\mathcal{F}_{0,q}=\{F\in\mathcal{F}_{0},q\in F\}$ and $\mathcal{F}_{0,0}=\{F\in\mathcal{F}_{0},q\notin F\}$. By~\cref{claim:NOcommon}, $|\mathcal{F}_{0,0}|>0$. Obviously we have $|\mathcal{F}_{0,q}|\ge 2$, otherwise $|\mathcal{F}_{0,0}|=0$, a contradiction. Moreover if $|\mathcal{F}_{0,q}|=2$, then $3|\mathcal{F}_{0}|\le \sum\limits_{v\in [n]\setminus\{p\}}|\mathcal{F}_{0,q}|\le 2(n-1)$, which implies $|\mathcal{F}_{0}|\le\frac{2(n-1)}{3}<n-2$ when $n\ge 7$, a contradiction to $|\mathcal{F}_{0}|\ge n-2$. Then we divide our argument according to $|\mathcal{F}_{0,q}|$: 
 \begin{Case}
    \item If $|\mathcal{F}_{0,q}|=3$, we arbitrarily pick a set $A:=\{a_{1},a_{2},a_{3}\}\in\mathcal{F}_{0,0}$, denote $\mathcal{F}_{0,q}$ to be $A_{1}\cup A_{2}\cup A_{3}$, observe that the sets in $\mathcal{F}_{1,q}=\{F\in\mathcal{F}_{1},q\notin F\}$ are of the form $\{p,a_{i},*\}$ for some $i\in [3]$. If there exists some $a_{i}$ appearing exactly twice in $\mathcal{F}_{0,q}$, assume $a_{i}=a_{1}$ by symmetry, then we denote $A_{1}=\{q,a_{1},b_{1}\}$, $A_{2}=\{q,a_{1},b_{2}\}$ and $A_{3}=\{q,a_{2},b_{3}\}$, in particular, $b_{3}\neq a_{1}$. Then $\{p,a_{1},*\}$ can only be $\{p,a_{1},a_{2}\}$ or $\{p,a_{1},b_{3}\}$, and $\{p,a_{2},*\}$ can only be $\{p,a_{1},a_{2}\}$. We then consider the possibilities of $\{p,a_{3},*\}$.
    \begin{itemize}
        \item If $b_{1}=a_{3}$ or $b_{2}=a_{3}$, without loss of generality, assume $b_{1}=a_{3}$, then if $b_{3}=a_{3}$, then $\{p,a_{3},*\}$ has to be $\{p,a_{3},a_{1}\}$ or $\{p,a_{3},b_{2}\}$, in this case $\{p,a_{3},a_{1}\}=\{p,a_{1},b_{3}\}$, which implies $|\mathcal{F}_{1,q}|\le 3$. If $b_{3}\neq a_{3}$, then the element $*$ of $\{p,a_{3},*\}$ belongs to $A_{2}\cap A_{3}$, which implies that either $\{p,a_{3},*\}=\{p,a_{3},b_{2}\}$ or $\{p,a_{3},*\}=\emptyset$. Therefore, in this case we have $|\mathcal{F}_{1,q}|\le 3$.
        \item If $a_{3}\notin \{b_{1},b_{2}\}$, then if $b_{3}=a_{3}$, then the element $*$ of $\{p,a_{3},*\}$ belongs to $A_{1}\cap A_{2}$, which yields that $\{p,a_{3},*\}$ has to be $\{p,a_{3},a_{1}\}$. Therefore, in this case we have $|\mathcal{F}_{1,q}|\le 3$. If $b_{3}\neq a_{3}$, it is easy to see that $\{p,a_{3},*\}=\emptyset$ or $\{p,a_{3},*\}=\{p,a_{3},b_{1}\}$ when $b_{1}=b_{2}=b_{3}$, which yields that $|\mathcal{F}_{1,q}|\le 3$ in this case.
    \end{itemize}
In all, we have $|\mathcal{F}_{1,q}|\le 3$. If each $a_{i}$ appears exactly once in $\mathcal{F}_{0,q}$, then we denote $A_{1}=\{q,a_{1},b_{1}\}$, $A_{2}=\{q,a_{2},b_{2}\}$ and $A_{3}=\{q,a_{3},b_{3}\}$ with $\{b_{1},b_{2},b_{3}\}\cap\{a_{1},a_{2},a_{3}\}=\emptyset$. Observe that if $b_{1},b_{2},b_{3}$ are pairwise distinct, then $|\mathcal{F}_{1,q}|=0$. By symmetry if $b_{1}=b_{2}\neq b_{3}$, then $\mathcal{F}_{1,q}$ can only consist of $\{p,a_{3},b_{1}\}$. In particular, if $b_{1}=b_{2}=b_{3}$, then $\mathcal{F}_{1,q}\subseteq\{p,a_{1},b_{1}\} \cup \{p,a_{2},b_{1}\} \cup \{p,a_{3},b_{1} \}$. Therefore, $|\mathcal{F}_{1,q}|\le 3$. Moreover, we can see $\mathcal{F}_{1}\setminus\mathcal{F}_{1,q}\subseteq \{p,q,a_{1}\}\cup \{p,q,a_{2}\}\cup \{p,q,a_{3}\}$, which together implies that $|\mathcal{F}_{1}|\le 6$. Note that $3|\mathcal{F}_{0}|\le \sum\limits_{v\in [n]\setminus\{p\}}|\mathcal{F}_{0,q}|\le 3(n-1)$, therefore, $|\mathcal{F}|= |\mathcal{F}_{0}|+|\mathcal{F}_{1}|\le n+5\le 2n-3$ when $n\ge 8$. In particular, when $n=7$, if $|\mathcal{F}_{0,0}|=1$, then $|\mathcal{F}|\le 1+3+6=10\le 11$, and if $|\mathcal{F}_{0,0}|\ge 2$, then $|\mathcal{F}_{1}\setminus \mathcal{F}_{1,q}|\le 2$, which implies that $|\mathcal{F}_{1}|\le 5$, then $|\mathcal{F}|\le|\mathcal{F}_{0}|+5\le 6+5=11<12$. 
    \item If $|\mathcal{F}_{0,q}|\ge 4$, then we directly apply the argument in \textbf{Subsubcase 1.3.1}, we can see $|\mathcal{F}_{1,q}|\le 3$. Then the degree of element $q$ is larger than the degree of element $p$, a contradiction. 
 \end{Case}

 In all, when $|\mathcal{F}_{0}|\ge n-2$, $|\mathcal{F}|\le 2n-3$.
 
\end{Case}
       
Thus when $k=3$, $|\mathcal{F}|\le 2n-2$, the equality holds if and only if $\mathcal{F}$ is isomorphic to $\mathcal{M}_{3,2}$.

        \item When $k=4$, by~\cref{claim:SuanTMD}, $h(2)=f(2)=f(n-4)$, then $|\mathcal{F}|\le \binom{n-1}{3}-h(2)= \binom{n-1}{3}-\binom{n-5}{3}-\binom{n-6}{2}+2$. Then there are different types of extremal configurations.
        \begin{Case}
       \item When $\mathcal{F}_{0}=\{F_{1},F_{2}\}$, by~\cref{claim:StabilityForS}, $|F_{1}\cap F_{2}|=3$. We set $F_{1}\cap F_{2}=\{a_{1},a_{2},a_{3}\}$, $F_{1}\setminus F_{2}=\{b_{1}\}$ and $F_{2}\setminus F_{1}=\{b_{2}\}$. Then for any $F\in\mathcal{F}_{1}$, if $F\cap \{a_{1},a_{2},a_{3}\}=\emptyset$, then $\{b_{1},b_{2}\}\subseteq F$, therefore, $\mathcal{F}_{1}=\{G\in \binom{[n]}{4}:\{b_{1},b_{2},p\}\in G\}\cup \{G\in \binom{[n]}{4}:p\in G,G\cap \{a_{1},a_{2},a_{3}\}\neq\emptyset\}$. Moreover, $\mathcal{F}_{0}=\{G\in\binom{[n]}{4}:\{a_{1},a_{2},a_{3}\}\subseteq G,G\cap\{b_{1},b_{2}\}\neq\emptyset\}$. Therefore, $\mathcal{F}$ is isomorphic to $\mathcal{M}_{k,2}$. 
       \item   $|\mathcal{F}_{0}|=n-4\ge 4$, $d_{\textup{max}}(n-4)\le \binom{n-1}{3}-\sum\limits_{i=1}^{4}\binom{n-4-i}{4-i}=d_{\textup{max}}(\mathcal{M}_{4})$, then by~\cref{thm:Frankl}, $|\mathcal{F}|\le|\mathcal{M}_{4}|$. Moreover, the equality holds if and only if $\mathcal{F}$ is either isomorphic to $\mathcal{M}_{3}$, or isomorphic to $\mathcal{M}_{4}$. Furthermore, by checking the structure of $\mathcal{M}_{3}$, we can see $\mathcal{F}$ has to be isomorphic to $\mathcal{M}_{4}$ in this case.
       \end{Case}
        \item When $k\ge 5$, by~\cref{claim:SuanTMD}, $h(2)=f(2)$, then $|\mathcal{F}|\le \binom{n-1}{k-1}-h(2)= \binom{n-1}{k-1}-\binom{n-k-1}{k-1}-\binom{n-k-2}{k-2}+2$. Using the almost identical argument as that in Subcase 2.1, we can show the above equality holds if and only if $\mathcal{F}$ is isomorphic to $\mathcal{M}_{k,2}$, we omit the repeated details here. 
    \end{Case}  
    In all, when $k=4$, $|\mathcal{F}|\le\binom{n-1}{3}-\binom{n-5}{3}-\binom{n-6}{2}+2$, and the equality holds if and only if $\mathcal{F}$ is isomorphic to $\mathcal{M}_{4,2}, \mathcal{M}_3$ or $\mathcal{M}_4$. When $k\ge 5$ or $k=3$, we have $|\mathcal{F}|\le\binom{n-1}{k-1}-\binom{n-k-1}{k-1}-\binom{n-k-2}{k-2}+2$, and the equality holds if and only if $\mathcal{F}$ is isomorphic to $\mathcal{M}_{k,2}$. This finishes the proof.
    \end{document}